\documentclass[11pt]{article}

\usepackage{amssymb,amsmath,euscript,mathrsfs,dsfont}
\usepackage{amsmath}

\usepackage{times}
\usepackage{bm}
\usepackage[natbib=true,
backend=bibtex,
maxbibnames=99,
maxcitenames=2,
citestyle=numeric,
bibstyle=authoryear,
firstinits=true]{biblatex}
\usepackage{booktabs}  
\usepackage{threeparttable}
\usepackage[plain,noend]{algorithm2e}
\usepackage{euscript}
\usepackage{bbm}
\usepackage{verbatim}
\usepackage{graphicx}
\usepackage{epstopdf}
\usepackage{color}
\usepackage{rotating}
\usepackage{multirow}
\usepackage[width=\textwidth]{caption}
\usepackage{hyperref}
\usepackage[dvipsnames]{xcolor} %
\newtheorem{proposition}{Proposition}[section]
\newtheorem{lemma}[proposition]{Lemma}
\newtheorem{theorem}[proposition]{Theorem}

\newtheorem{algo}[proposition]{Algorithm}

\newcommand{\R}{\mathbb{R}}

\newcommand{\mH}{\mathcal{H}}

\renewcommand{\P}{{\rm P}}
\newcommand{\E}{{\rm E}}

\newcommand{\FDR}{{\rm FDR}}
\newcommand{\Power}{{\rm Power}}

\newcommand{\BH}{{\rm BH}}

\newcommand{\SNR}{{\rm SNR}}

\makeatletter \@addtoreset{equation}{section} \makeatother
\newcommand {\qed}%
{%
    {}\hfill
    {}\hfill
    {$\square $}%
    \vspace {0.3cm}%
    \pagebreak [2]%
    \par
}%
\newenvironment{proof}{%
    \vspace{0.3cm}%
    \pagebreak [2]%
    \par%
    \noindent {\bf  Proof.}}{\qed}%
\newenvironment{example}{%
    \vspace{0.3cm} \pagebreak [2]%
    \par%
    \refstepcounter{proposition}%
    \noindent%
    {\bf  Example~\theproposition\ }}{}%
\DefineBibliographyExtras{english}{}
\AtEveryBibitem{%
  \clearfield{number}}
\DeclareNameAlias{sortname}{last-first}
\DeclareFieldFormat[article,periodical]{volume}{\mkbibbold{#1}}
\renewbibmacro{in:}{}
\DeclareFieldFormat{pages}{#1}
\DeclareFieldFormat[article,inbook,inproceedings]{title}{#1} 
\DeclareFieldFormat{labelnumberwidth}{\mkbibbrackets{#1}}
\defbibenvironment{bibliography}
  {\list
     {\printtext[labelnumberwidth]{%
    \printfield{prefixnumber}%
    \printfield{labelnumber}}}
     {\setlength{\labelwidth}{\labelnumberwidth}%
      \setlength{\leftmargin}{\labelwidth}%
      \setlength{\labelsep}{\biblabelsep}%
      \addtolength{\leftmargin}{\labelsep}%
      \setlength{\itemsep}{\bibitemsep}%
      \setlength{\parsep}{\bibparsep}}%
      }
  {\endlist}
  {\item}
\begin{document}

\title{Multiple Testing of Local Extrema for Detection of Change Points}
\author{Dan Cheng\thanks{Research partially
		supported by NSF grant DMS-1902432.}\\ Arizona State University
	\and Zhibing He\footnotemark[1]\\ Arizona State University
	\and Armin Schwartzman\thanks{Research partially
		supported by NSF grant DMS-1811659.} \\ University of California San Diego
	 }

\date{}

\maketitle

\begin{abstract}
A new approach to detect change points based on differential smoothing and multiple testing is presented for long data sequences modeled as piecewise constant functions plus stationary ergodic Gaussian noise. As an application of the STEM algorithm for peak detection developed in \protect\citet{schwartzman2011multiple} and \protect\citet{cheng2017multiple}, the method detects change points as significant local maxima and minima after smoothing and differentiating the observed sequence. The algorithm, combined with the Benjamini-Hochberg procedure for thresholding p-values, provides asymptotic strong control of the False Discovery Rate (FDR) and power consistency, as the length of the sequence and the size of the jumps get large. Simulations show that FDR levels are maintained in non-asymptotic conditions and guide the choice of smoothing bandwidth. The methods are illustrated in magnetometer sensor data and genomic array-CGH data. An R package named ``dSTEM" is available in R cran.
\end{abstract}


\section{Introduction}
\label{sec:intro}
Detecting change points in the mean of an observed signal is a common statistical problem with applications in many research areas such as climatology \protect\citep{reeves2007review}, oceanography\protect\citep{killick2010detection}, finance \protect\citep{zeileis2010testing} and medical imaging \protect\citep{nam2012quantifying}. It often appears in the analysis of time series but it has more recently been found in the analysis of genomic sequences, see \protect\citet{erdman2008fast, lai2005comparative, muggeo2010efficient, olshen2004circular, tibshirani2007spatial, wang2005method} and the references therein. Given the large amounts of data present in modern applications, it is of interest to design a change point detection method that can operate over long sequences where the number and location of change points are unknown, and in such a way that the overall detection error rate is controlled.

Many different approaches have been proposed to find and estimate change points, such as kernel-based methods \protect\citep{akritas2003inference}, Bayesian methods \protect\citep{barry1993bayesian, erdman2008fast}, segmentation techniques \protect\citep{olshen2004circular, wang2005method, muggeo2010efficient}, nonparametric tests \protect\citep{lanzante1996resistant} and $L_1$-penalty methods \protect\citep{eilers2004quantile, huang2005detection, tibshirani2007spatial}, including the PELT method \protect\citep{jackson2005algorithm, killick2010detection}. Though there is abundant literature on change points segmentation and detection, only a few papers address the FDR issue which treats the detection of change points as multiple hypothesis testing problems. \protect\citet{tibshirani2007spatial} applied the fused lasso method to the hot-spot detection problem, and provided empirical evidence for the FDR control. \protect\citet{efron2011false} introduced an iterative local FDR based algorithm to explore copy number changes.
Recently, \protect\citet{frick2014multiscale} introduced a simultaneous multiscale change point estimator (SMUCE) for the change point problem in exponential family regression, and proved the control of the probability of overestimating the true number of change points. \protect\citet{li2016fdr} improved the SMUCE method and proposed a multiscale segmentation method FDRSeg, which gives a non-asymptotic upper bound for its FDR in a Gaussian setting and is robust to the choice of parameter $\alpha$.
However, our proposed approach is unique in the following two ways.

First, the noise is assumed to be a stationary Gaussian process, allowing the error terms to be correlated. This is an important departure from the standard assumption of white noise in most of the change-point literature. In fact, applied statisticians desiring to use change-point methods have sometimes abandoned this option in favor of other techniques simply because the white noise assumption does not hold \protect\citep{hung2013hidden}. This paper shows that change-point methods can be devised for correlated noise, expanding the domain of their applicability.

Second, we use the theory of Gaussian processes to compute p-values for all candidate change points, so that significant change points can be selected at a desired significance level. For concreteness, we adopt the Benjamini-Hochberg multiple testing procedure, enabling control of the false discovery rate (FDR) of detected change points when the data sequence is long and the number and location of change points are unknown. To our knowledge, this is the first article proposing a multiple testing method for controlling the FDR of detected change points. Moreover, the asymptotic properties of FDR and power are provided.

In this paper, we consider a signal-plus-noise model where the true signal is a piecewise constant function and the change points are defined as the points of discontinuity. Inspired by the method for detecting peaks in \protect\citet{schwartzman2011multiple} and \protect\citet{cheng2017multiple}, we modify the STEM algorithm therein to detect change points. The central idea is the observation that the true signal has zero derivative everywhere except at the change points, where the derivative is infinite. Thus, in the presence of noise and under temporal or spatial sampling, change points can be seen as positive or negative peaks in the derivative of the smoothed signal. Note that because of the time sampling, derivatives cannot be observed directly and can only be estimated. The focus on the derivative of the smoothed signal effectively transforms the change point detection problem into a peak detection problem. As in the STEM algorithm, the resulting peak detection problem is then solved by identifying local maxima and minima of the derivative as candidate peaks and applying a multiple testing procedure to the list of candidates.

The ``differential Smoothing and TEsting of Maxima/Minima" (dSTEM) algorithm for change point detection consists of the following steps:
\begin{enumerate}
\item {\em Differential kernel smoothing}: to transform change points to local maxima or minima, and to increase the \SNR. This principle of this step is illustrated in Figure \ref{fig:hr}.
\item {\em Candidate peaks}: find local maxima and minima of the differentiated smoothed process.
\item {\em P-values}: computed at each local maximum and minimum under the the null hypothesis of no signal in a local neighborhood.
\item {\em Multiple testing}: apply a multiple testing procedure to the set of local maxima and minima; declare as detected change points those local maxima and minima whose p-values are significant.
\end{enumerate}

The algorithm is illustrated by a toy example in Figure \ref{fig:simul}.

\begin{figure}[t]
\centering
\includegraphics[width=4.8in,height=2.2in]{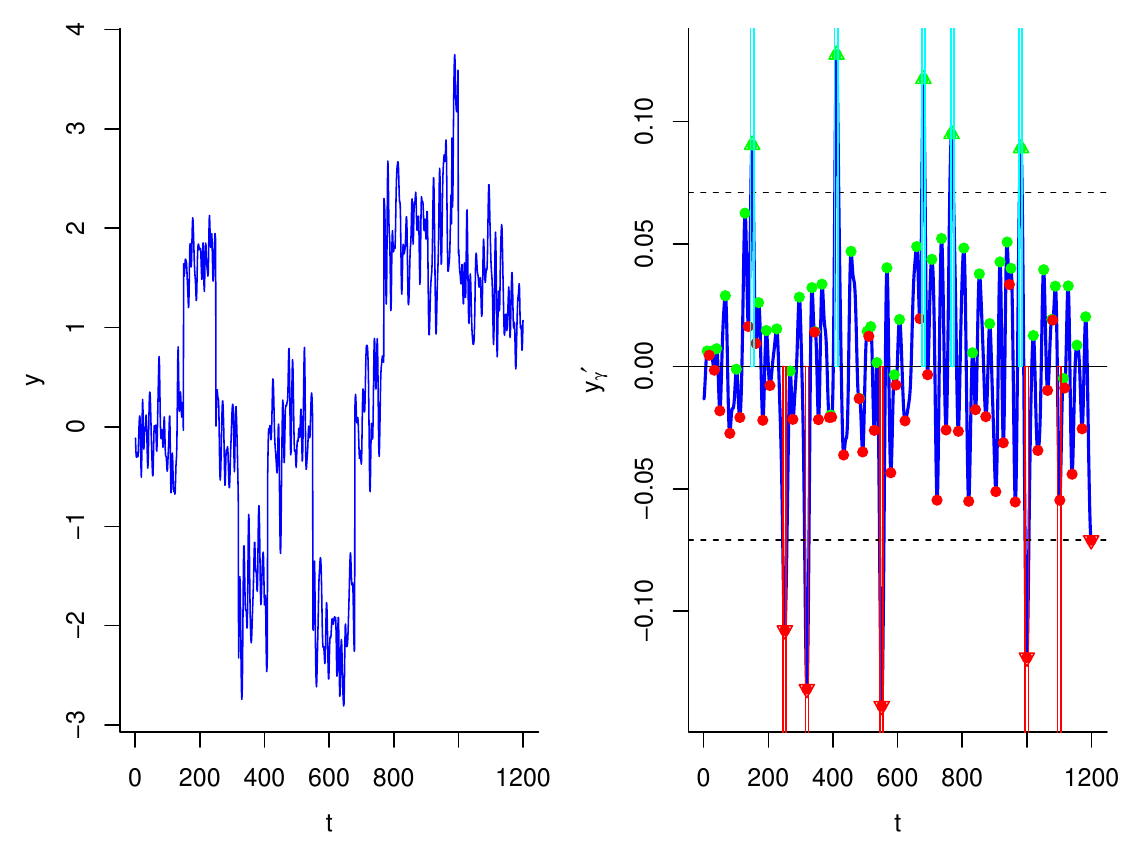} 
\caption{Following the notation in \S \ref{sec:model} and Example \ref{ex:Gaussian}, the left panel is the observed signal-plus-noise model $y(t)$ containing ten true change points with varying $a_i$ and noise $z(t)$ given by \eqref{eq:noiseEx} with $\sigma=1$ and $\nu=2$. The right panel illustrates the dSTEM algorithm. The blue curve is $y_\gamma'(t)$, obtained with a Gaussian smoothing kernel with standard deviation $\gamma=6$. Local maxima (green solid dots) and local minima (red solid dots) are declared as significant (marked with solid triangles) at FDR level $\alpha=0.1$ if their heights are beyond the dotted line thresholds. The cyan and pink bars indicate the location tolerance intervals $(v_i-b, v_i+b)$ with $b=5$ for increasing and decreasing change points respectively. At this tolerance, there are nine true discoveries and one false discovery.}
\label{fig:simul}
\end{figure}

\begin{figure}[!h]
\centering
\includegraphics[width= 4.8in,height=2.2in]{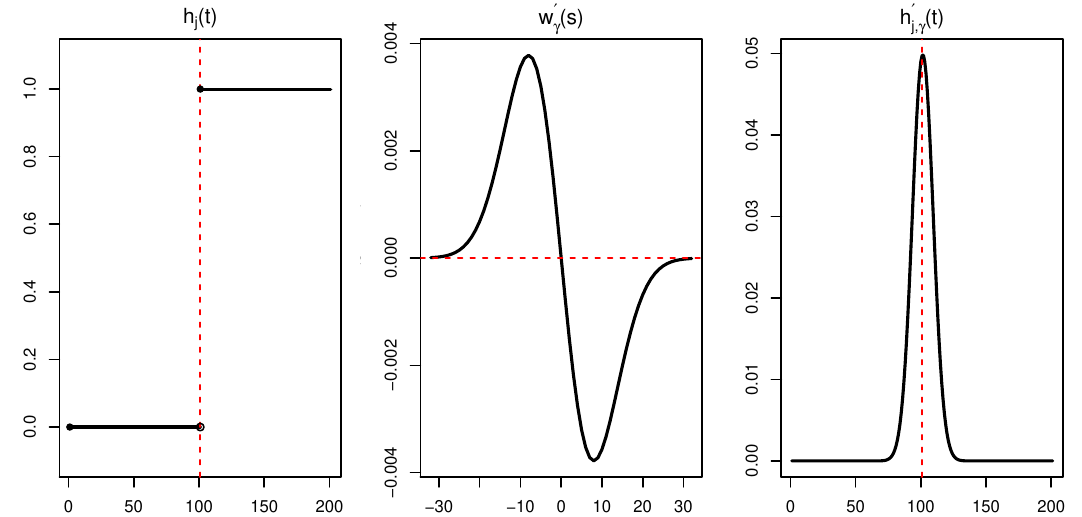}
\caption{Following the notation in \S \ref{sec:model} and \S \ref{sec:model_dev}, the left panel is the change point indicator function $h_j(t)$ with $v_j = 101$, that is, there exists only one change point at the location $t=v_j = 101$. The middle panel is the differential Gaussian kernel $w_{\gamma}^{'}(s)$ on $[-4\gamma,4\gamma]$ with $\gamma=8$. The right panel is $h_{j,\gamma}^{'}(t)$, which is the differential kernel smoothing (under Gaussian kernel) of $h_{j}(t)$ obtained by convolution, as shown in \eqref{eq:h-diff}. By such transformation, the change point becomes a local maxima of a smooth function with compact support.}
\label{fig:hr}
\end{figure}

The dSTEM algorithm above differs from the ones in \protect\citet{schwartzman2011multiple} and \protect\citet{cheng2017multiple} in that peaks are sought in the derivative of the smoothed signal rather than the smoothed signal itself, and that both positive and negative peaks are considered. In addition, an important consideration for the proper definition of error in change point detection is that, as opposed to the peak detection problems considered in \protect\citet{schwartzman2011multiple} and \protect\citet{cheng2017multiple} where signal peaks had compact support, a true single change point over a continuous domain at $t = v$ has Lebesgue measure zero. Thus in the presence of noise, it can hardly be detected exactly at $t = v$. Therefore we introduce a \emph{location tolerance} $b$ that defines the precision within which a change point should be detected. Specifically, given $b$, a detected change point is regarded as a true discovery if it falls in the interval $(v - b, v + b)$. Conversely, if a significant change point is found more than a distance $b$ from any true change point, it is considered a false discovery. The quantity $b$ is not used in the dSTEM algorithm itself but is needed for proper error definition.

Under this convention, it is shown here that the dSTEM algorithm exhibits asymptotic FDR control and power consistency as the length of the sequence and the size of the jumps at the change points increase. These asymptotic conditions are similar to those considered in \protect\citet{schwartzman2011multiple} and \protect\citet{cheng2017multiple} and, in fact, the proofs are extended from those in \protect\citet{cheng2017multiple}.

Simulations for varying levels of smoothing bandwidth $\gamma$, smoothing degree of noise $\nu$ and jump size $a$ are used to study the behavior of the algorithm under non-asymptotic conditions. The simulation results help guide the choice of smoothing bandwidth $\gamma$ with respect to $\nu$ and the desired location tolerance. In general, power increases with bandwidth to a limit dictated by the distance between the change points, so admitting a higher tolerance generally allows a higher bandwidth and higher power.

The methods are illustrated in a genomic sequence of array-CGH data in a breast-cancer tissue sample \protect\citep{loo2004array,hsu2005denoising}. The goal of the analysis is to find genomic segments with copy-number alterations. These are found by detecting change points in the copy number genomic sequence.
Another application is magnetometer sensor readings, aiming at finding the start and end points of hand gesture motion, which is the critical step and foundation of establishing a secret key based on hand gestures.

\section{The multiple testing scheme}
\label{sec:model_algorithm}

\subsection{The model}
\label{sec:model}
We consider a continuous time model, although the algorithm is designed for data discretely sampled in time. Consider the signal-plus-noise model
\begin{equation}
\label{eq:signal+noise}
y(t) = \mu(t) + z(t), \qquad t \in \R,
\end{equation}
where the signal $\mu(t)$ is a piecewise constant function of the form
\begin{equation*}
\mu(t) = \sum_{j=0}^\infty a_j h_j(t), \qquad a_j \in \R\setminus \{0\},
\end{equation*}
with $h_j(t)=\mathbbm{1}(t\geq v_j)$ for $v_j\in \R$. We are interested in finding the \emph{change points} $v_j$. For the asymptotic analysis, we assume
\begin{equation}\label{eq:a-v}
a= \inf_{j}|a_j|>0 \quad {\rm and } \quad d=\inf_{j}|v_j-v_{j-1}|>0,
\end{equation}
so that the change points do not become arbitrarily small in size nor arbitrarily close to each other.

Let $w_\gamma(t)=w(t/\gamma)/\gamma$, where $\gamma>0$ is the bandwidth parameter and $w(t) \ge 0$ is a unimodal symmetric kernel with compact connected support $[-c,c]$ and unit action. Convolving the process \eqref{eq:signal+noise} with the kernel $w_\gamma(t)$ results in the smoothed random process
\begin{equation}
\label{eq:conv}
y_\gamma(t) = w_\gamma(t) * y(t) =
\int_{\R} w_\gamma(t-s) y(s)\,ds = \mu_\gamma(t) + z_\gamma(t),
\end{equation}
where the smoothed signal and smoothed noise are defined respectively as
\begin{equation}
\label{eq:mu-gamma}
\mu_\gamma(t) = w_\gamma(t) * \mu(t) = \sum_{j=0}^\infty a_j h_{j,\gamma}(t) \quad {\rm and} \quad z_\gamma(t) = w_\gamma(t) * z(t),
\end{equation}
and where the smoothed change point takes the form
\begin{equation}
\label{eq:h}
h_{j,\gamma}(t) = w_\gamma(t) * h_j(t).
\end{equation}
The smoothed noise $z_\gamma(t)$ defined by \eqref{eq:mu-gamma} is assumed to be a zero-mean four-times differentiable stationary ergodic Gaussian process.


\subsection{Change point detection as peak detection of the derivative}
\label{sec:model_dev}
Consider now the derivative of the smoothed observed process \eqref{eq:conv}
\begin{equation}
\label{eq:conv-diff}
y'_\gamma(t) = w'_\gamma(t) * y(t) =
\int_{\R^N} w'_\gamma(t-s) y(s)\,ds = \mu'_\gamma(t) + z'_\gamma(t),
\end{equation}
where the derivatives of the smoothed signal and smoothed noise are respectively
\begin{equation*}
\mu'_\gamma(t) = w'_\gamma(t) * \mu(t) = \sum_{j=0}^\infty a_j h'_{j,\gamma}(t) \quad {\rm and} \quad z'_\gamma(t) = w'_\gamma(t) * z(t).
\end{equation*}
A key observation from \eqref{eq:h} is that
\begin{equation}\label{eq:h-diff}
\begin{split}
h'_{j,\gamma}(t) &= \int_{\R} w'_\gamma(t-s) h_j(s)\,ds = \int_{\R} w'_\gamma(s) h_j(t-s)\,ds\\
& = \int_{\R} w'_\gamma(s) \mathbbm{1}(t-s\geq v_j)\,ds = \int_{-\infty}^{t-v_j}w'_\gamma(s)\,ds = w_\gamma(t-v_j),
\end{split}
\end{equation}
as illustrated in Figure \ref{fig:hr}. Thus \eqref{eq:conv-diff} represents a signal-plus-noise model where the smoothed signal
\begin{equation}
\label{eq:mu-gamma-diff-w}
\mu'_\gamma(t) = \sum_{j=0}^\infty a_j h'_{j,\gamma}(t) = \sum_{j=0}^\infty a_j w_\gamma(t-v_j)
\end{equation}
is a sequence of unimodal peaks with the same shape as that of $w_\gamma$ and located at locations $v_j$. The problem of finding change points in $y_\gamma(t)$ is thus reduced to finding (positive or negative) peaks in $y'_\gamma(t)$.

For simplicity, we assume that the compact supports $S_{j,\gamma}$ of the smoothed peak shape $h'_{j,\gamma}(t) = w_\gamma(t-v_j)$ do not overlap, although this is not crucial in practice.


\subsection{The dSTEM algorithm for change point detection}
\label{sec:alg}
Suppose we observe $y(t)$ with $J$ jumps defined by \eqref{eq:signal+noise} in the line of length $L$ centered at the origin, denoted by $U(L)=(-L/2,L/2)$. The following dSTEM (differential Smoothing and TEsting of Maxima/Minima) is a modified version of the STEM algorithm of \protect\citet{schwartzman2011multiple} and \protect\citet{cheng2017multiple} for detecting change points.

\begin{algo}[dSTEM algorithm for change point detection]
\label{alg:STEM}
\hfill\par\noindent
\begin{enumerate}
\item {\em Differential kernel smoothing}:
Obtain the process \eqref{eq:conv-diff} by convolution of $y(t)$ with the kernel derivative $w'_\gamma(t)$.
\item {\em Candidate peaks}:
Find the set of local maxima and minima of $y'_\gamma(t)$ in $U(L)$, denoted by $\tilde{T}_\gamma = \tilde{T}^+_{\gamma} \cup \tilde{T}^-_{\gamma}$, where
\begin{equation*}
\begin{split}
\tilde{T}^+_{\gamma} &= \left\{ t \in U(L): y''_{\gamma}(t) = 0, \  y^{(3)}_{\gamma}(t) < 0\right\},\\
\tilde{T}^-_{\gamma} &= \left\{ t \in U(L): y''_{\gamma}(t) = 0, \  y^{(3)}_{\gamma}(t) > 0\right\}.
\end{split}
\end{equation*}
\item {\em P-values}:
For each $t \in \tilde{T}^+_\gamma$, compute the p-value $p_\gamma(t)$ for testing the (conditional) hypotheses
$$
\begin{aligned}
\mH_{0}(t):& \ \{\mu'(s) = 0 \quad \text{for all} \quad s \in (t-b,t+b) \} \quad {\rm vs.} \\
\mH_{A}(t):& \ \{\mu'(s) > 0 \quad \text{for some} \quad s \in (t-b,t+b) \};
\end{aligned}
$$
and for each $t \in \tilde{T}^-_\gamma$, compute the p-value $p_\gamma(t)$ for testing the (conditional) hypotheses
$$
\begin{aligned}
\mH_{0}(t):& \ \{\mu'(s) = 0 \quad \text{for all} \quad s \in (t-b,t+b) \} \quad {\rm vs.} \\
\mH_{A}(t):& \ \{\mu'(s) < 0 \quad \text{for some} \quad s \in (t-b,t+b) \},
\end{aligned}
$$
where $b>0$ is an appropriate location tolerance.

\item {\em Multiple testing}:
Let $\tilde{m}_\gamma = \# \{t\in \tilde{T}_\gamma\}$ be the number of tested hypotheses. Apply a multiple testing procedure on the set of $\tilde{m}_\gamma$ p-values $\{p_\gamma(t), \, t \in \tilde{T}_\gamma\}$, and declare significant all local extrema whose p-values are smaller than the significance threshold.
\end{enumerate}
\end{algo}

\subsection{P-values}
\label{sec:pvalue}
Given the observed heights $y'_\gamma(t)$ at the local maxima or minima $t\in\tilde{T}_\gamma = \tilde{T}^+_{\gamma} \cup \tilde{T}^-_{\gamma}$, p-values in step (3) of Algorithm \ref{alg:STEM} are computed as
\begin{equation}
\label{eq:pval}
p_\gamma(t) =\begin{cases}
F_\gamma(y'_\gamma(t)), \quad t \in \tilde{T}^+_{\gamma}, \\
F_\gamma(-y'_\gamma(t)), \quad t \in \tilde{T}^-_{\gamma},
\end{cases}
\end{equation}
where $F_\gamma(u)$ denotes the right tail probability of $z'_\gamma(t)$ at the local maximum $t \in \tilde{T}^+_{\gamma}$, evaluated under the null model $\mu'(s) = 0, \forall s\in (t-b, t+b)$, that is,
\begin{equation}
\label{eq:palm}
F_\gamma(u) = \P\left(z'_\gamma(t) > u \big| \text{ $t$ is a local maximum of $z'_\gamma(t)$}\right).
\end{equation}
The second line in \eqref{eq:pval} is obtained by noting that, by \eqref{eq:palm},
\begin{equation*}
\label{eq:palm2}
\begin{split}
\P&\left(z'_\gamma(t) < u \big| \text{ $t$ is a local minimum of $z'_\gamma(t)$}\right)\\
& = \P\left(-z'_\gamma(t) > -u \big| \text{ $t$ is a local maximum of $-z'_\gamma(t)$}\right) = F_\gamma(-u),
\end{split}
\end{equation*}
since $-z'_\gamma(t)$ and $z'_\gamma(t)$ have the same distribution.

The distribution \eqref{eq:palm} has a closed-form expression, which can be obtained as in \protect\citet{schwartzman2011multiple} or \protect\citet{cramer2013stationary}. More specifically, the distribution \eqref{eq:palm} is given by
\begin{equation}
\label{eq:distr}
F_\gamma(u) = 1 - \Phi\left(u \sqrt{\frac{\lambda_{6, \gamma}}{\Delta}}\right) + \sqrt{\frac{2\pi\lambda^2_{4, \gamma}}{\lambda_{6, \gamma}{\sigma'_\gamma}^2}}\phi\left(\frac{u}{\sigma'_\gamma}\right)\Phi\left( u \sqrt{\frac{\lambda^2_{4, \gamma}}{\Delta {\sigma'_\gamma}^2}}\right),
\end{equation}
where ${\sigma'_\gamma}^2={\rm Var}(z_\gamma'(t))$, $\lambda_{4, \gamma}={\rm Var}(z_\gamma^{''}(t))$, $\lambda_{6, \gamma}={\rm Var}(z_\gamma^{(3)}(t))$, $\Delta =  {\sigma'_\gamma}^2\lambda_{6,\gamma} - \lambda_{4,\gamma}^2$, and $\phi(x)$, $\Phi(x)$ are the standard normal density and cumulative distribution function, respectively. The quantities ${\sigma'_\gamma}^2$, $\lambda_{4, \gamma}$ and $\lambda_{6, \gamma}$ depend on the kernel $w_\gamma(t)$ and the autocorrelation function of the original noise process $z(t)$. Explicit expressions may be obtained, for instance, for the Gaussian
autocorrelation model in Example \ref{ex:Gaussian} below, which we use later in the simulations.

\subsection{Error definitions}
\label{sec:errors}
Assuming the model of \S \ref{sec:model}, define the \emph{signal region} $\mathbb{S}_1^b =  \cup_{j=1}^J (v_j-b, v_j+b)$ and \emph{null region} $\mathbb{S}_0^b = U(L) \setminus \mathbb{S}_1^b$. For $u>0$, let $\tilde{T}_\gamma(u) = \tilde{T}^+_{\gamma}(u) \cup \tilde{T}^-_{\gamma}(u)$, where
\begin{equation*}
\begin{split}
\tilde{T}^+_{\gamma}(u) = \left\{ t \in U(L): y_\gamma'(t)>u, \ y''_{\gamma}(t) = 0, \  y^{(3)}_{\gamma}(t) < 0\right\},\\
\tilde{T}^-_{\gamma}(u) = \left\{ t \in U(L): y_\gamma'(t)<-u, \ y''_{\gamma}(t) = 0, \  y^{(3)}_{\gamma}(t) > 0\right\},
\end{split}
\end{equation*}
indicating that $\tilde{T}^+_{\gamma}(u)$ and $\tilde{T}^-_{\gamma}(u)$ are respectively the set of local maxima of $y_\gamma'(t)$ above $u$ and the set of local minima of $y_\gamma'(t)$ below $-u$. The number of totally and falsely detected change points at threshold $u$ are defined respectively as
\begin{equation}
\label{eq:RV}
\begin{split}
R_\gamma(u) &= \#\{t\in \tilde{T}^+_\gamma(u)\} +  \#\{t\in \tilde{T}^-_\gamma(u)\},\\
V_\gamma(u; b) &=\#\{t\in \tilde{T}^+_\gamma(u)\cap \mathbb{S}_0^b\} + \#\{t\in \tilde{T}^-_\gamma(u)\cap \mathbb{S}_0^b\}.
\end{split}
\end{equation}
Both are defined as zero if $\tilde{T}_\gamma(u)$ is empty. The FDR at threshold $u$ is defined as the expected proportion of falsely detected jumps
\begin{equation}
\label{eq:FDR}
\FDR_\gamma(u; b) = \E\left\{ \frac{V_\gamma(u; b)}{R_\gamma(u)\vee1} \right\}.
\end{equation}
Note that when $\gamma$ and $u$ are fixed, $V_\gamma(u; b)$ and hence $\FDR_\gamma(u; b)$ are decreasing in $b$.

Following the notation in \protect\citet{cheng2017multiple}, define the \emph{smoothed signal region} $\mathbb{S}_{1, \gamma}$ to be the support of $\mu'_\gamma(t)$ and \emph{smoothed null region} $\mathbb{S}_{0, \gamma} = U(L) \setminus \mathbb{S}_{1, \gamma}$. We call the difference between the expanded signal support due to smoothing and the true signal support the \emph{transition region} $\mathbb{T}_{\gamma} = \mathbb{S}_{1, \gamma}\setminus \mathbb{S}_1^b =\mathbb{S}_0^b\setminus \mathbb{S}_{0, \gamma}$.


\subsection{Power}
\label{sec:power}
Denote by $I^+$ and $I^-$ the collections of indices $j$ corresponding to increasing and decreasing change points $v_j$, respectively. We define the power of Algorithm \ref{alg:STEM} as the expected fraction of true discovered change points
\begin{equation}
\label{eq:power}
\begin{split}
\Power_\gamma(u; b) = \frac{1}{J} \sum_{j=1}^{J} \Power_{j,\gamma}(u; b) &= \E \Bigg[ \frac{1}{J} \Bigg(\sum_{j\in I^+} \mathbbm{1}\left(
\tilde{T}^+_\gamma(u) \cap (v_j-b, v_j+b) \ne \emptyset \right)\\
&\quad + \sum_{j\in I^-} \mathbbm{1}\left(
\tilde{T}^-_\gamma(u) \cap (v_j-b, v_j+b) \ne \emptyset \right) \Bigg)\Bigg],
\end{split}
\end{equation}
where $\Power_{j,\gamma}(u; b)$ is the probability of detecting jump $v_j$ within a distance $b$,
\begin{equation}\label{eq:power-j}
\begin{split}
\Power_{j,\gamma}(u; b) = \left\{
  \begin{array}{ll}
    \P\left(\tilde{T}^+_\gamma(u)\cap (v_j-b, v_j+b) \ne \emptyset \right), & {\rm if }\ j\in I^+, \\
    \P\left(\tilde{T}^-_\gamma(u)\cap (v_j-b, v_j+b) \ne \emptyset \right), & {\rm if }\ j\in I^-.
  \end{array}
\right.
\end{split}
\end{equation}
The indicator function in \eqref{eq:power} ensures that only one significant local extremum is counted within a distance $b$ of a change point, so power is not inflated. Note that when $\gamma$ and $u$ are fixed, $\Power_\gamma(u; b)$ and $\Power_{j,\gamma}(u; b)$ are increasing in $b$.


\section{Asymptotic FDR control and power consistency}
\label{sec:error-power}
Suppose the BH procedure is applied in step 4 of Algorithm \ref{alg:STEM} as follows. For a fixed $\alpha \in (0,1)$, let $k$ be the largest index for which the $i$th smallest p-value is less than $i\alpha/\tilde{m}_\gamma$. Then the null hypothesis $\mH_0(t)$ at $t \in \tilde{T}_\gamma$ is rejected if
\begin{equation}
\label{eq:thresh-BH-random}
p_\gamma(t) < \frac{k\alpha}{\tilde{m}_\gamma} \iff
\begin{cases}
y'_\gamma(t) > \tilde{u}_{\BH} = F_\gamma^{-1} \left(\frac{k\alpha}{\tilde{m}_\gamma}\right) & \text{ if } t\in \tilde{T}^+_{\gamma}, \\ y'_\gamma(t) < -\tilde{u}_{\BH} = -F_\gamma^{-1} \left(\frac{k\alpha}{\tilde{m}_\gamma}\right) & \text{ if } t\in \tilde{T}^-_{\gamma}, \end{cases}
\end{equation}
where $k\alpha/\tilde{m}_\gamma$ is defined as 1 if $\tilde{m}_\gamma=0$. Since $\tilde{u}_{\BH}$ is random, we define FDR in such BH procedure as
\begin{equation*}
\FDR_{\BH, \gamma}(b) = \E\left\{ \frac{V_\gamma(\tilde{u}_{\BH}; b)}{R_\gamma(\tilde{u}_{\BH})\vee1} \right\},
\end{equation*}
where $R_\gamma(\cdot)$ and $V_\gamma(\cdot; b)$ are defined in \eqref{eq:RV} and the expectation is taken over all possible realizations of the random threshold $\tilde{u}_{\BH}$. We will make use of the following conditions:

\begin{enumerate}
\item[(C1)] The assumptions of \S \ref{sec:model} hold.
\item[(C2)] $L \to \infty$ and $a = \inf_j |a_j| \to\infty$, such that $(\log L)/a^2 \to 0$, $J/L = A_1 + O(a^{-2}+L^{-1/2})$ with $A_1>0$.
\end{enumerate}
In condition (C2), we assume that the length of the search space $L$ increases. In order for the detection procedure to have good power while the error is controlled, the signal strength $a$ should also increase. This assumption is not restrictive since $(\log L)/a^2 \to 0$ implies the search space may grow exponentially faster than the signal strength. These conditions are realistic in applications. In data with repeated observations
and sample size $n$, the signal-to-noise ratio (SNR), equivalent to our signal strength $a$ here, is proportional to $\sqrt{n}$, which is large when the sample size is large in the classical asymptotic sense. Setting $p=L$ as the dimensionality of the problem, the condition $(\log L)/a^2 \to 0$ becomes $(\log p)/n \to 0$, which is similar to the condition required for consistent model selection in high dimensional regression.

Let $\E[\tilde{m}_{0,\gamma}(U(1))]$ and $\E[\tilde{m}_{0,\gamma}(U(1),u)]$ be the expected number of local maxima and local maxima above level $u$ of $z'_\gamma(t)$ on the unit interval $U(1)=(-1/2, 1/2)$, respectively. In particular, we have the following explicit formula \protect\citep{schwartzman2011multiple}
\begin{equation}\label{eq:Expect-LocalMax}
\E[\tilde{m}_{0,\gamma}(U(1))] = \frac{1}{2\pi}\sqrt{\frac{\lambda_{6, \gamma}}{\lambda_{4, \gamma}}}.
\end{equation}
Note that, by symmetry, the expected number of local minima and local minina below level $u$ of $z'_\gamma(t)$ on the unit interval $U(1)$ are given by $\E[\tilde{m}_{0,\gamma}(U(1))]$ and $\E[\tilde{m}_{0,\gamma}(U(1), -u)]$, respectively.

\begin{theorem}
\label{thm:FDR}
Let conditions (C1) and (C2) hold.

(i) Suppose Algorithm \ref{alg:STEM} is applied with a fixed threshold $u>0$. Then
\begin{equation*}
\FDR_\gamma(u; b) \le \frac{2\E[\tilde{m}_{0, \gamma}(U(1),u)](1-2c\gamma A_1)}{2\E[\tilde{m}_{0, \gamma}(U(1),u)](1-2c\gamma A_1) + A_1} + O(a^{-2}+L^{-1/2}).
\end{equation*}

(ii) Suppose Algorithm \ref{alg:STEM} is applied with the random threshold $\tilde{u}_{\BH}$ \eqref{eq:thresh-BH-random}. Then
\begin{equation*}
\FDR_{\BH, \gamma}(b) \le \alpha\frac{2\E[\tilde{m}_{0, \gamma}(U(1))](1-2c\gamma A_1)}{2\E[\tilde{m}_{0, \gamma}(U(1))](1-2c\gamma A_1) + A_1} + O(a^{-1}+L^{-1/4}).
\end{equation*}
\end{theorem}

\begin{proof}
Since $w_\gamma(t)$ has compact support $[-c\gamma, c\gamma]$, by \eqref{eq:h-diff}, the support $\mathbb{S}_{1,\gamma}$ of $\mu'_\gamma(t)$ in \eqref{eq:mu-gamma-diff-w} is composed of the support segments $[v_j-c\gamma, v_j+c\gamma]$ of $h'_{j,\gamma}(t)$. By condition (C2), $|\mathbb{S}_{1,\gamma}|/L = 2c\gamma A_1 + O(a^{-2}+L^{-1/2})$, which implies $|\mathbb{S}_{0,\gamma}|/L = 1-2c\gamma A_1 + O(a^{-2}+L^{-1/2})$.

Notice that, on the null region $\mathbb{S}_{0,\gamma}$, the expected number of local extrema, including both local maxima and minima, equals $2|\mathbb{S}_{0,\gamma}|\E[\tilde{m}_{0, \gamma}(U(1))]$. On the other hand, following the proof of Theorem 3 in \protect\citet{cheng2017multiple}, the expected number of local extrema on the signal region $\mathbb{S}_{1,\gamma}$ is asymptotically equivalent to $J$. This is because, for each $j\in I^+$ and $b>0$, as $a\to \infty$, asymptotically, there is no local maximum of $y'_\gamma(t)$ in $(v_j-c\gamma, v_j-b)\cup(v_j+b, v_j+c\gamma)$, and there is only one local maximum of $y'_\gamma(t)$ in $(v_j-b, v_j+b)$. The reasoning for the case of minima is similar.

The result then follows from the same arguments for proving Theorem 3 in \protect\citet{cheng2017multiple} with $N=1$, $A_{2,\gamma} = 2c\gamma A_1$, $z_\gamma(t)$ replaced by $z_\gamma'(t)$ and $\E[\tilde{m}_{0, \gamma}(U(1))]$ replaced by $2\E[\tilde{m}_{0, \gamma}(U(1))]$.
\end{proof}

\begin{lemma} \label{lemma:power-approx}
Let conditions (C1) and (C2) hold. As $|a_j| \to \infty$, the power for peak $j$ and fixed $u$ \eqref{eq:power-j} can be approximated by
\begin{equation}
\label{eq:approx-power}
\Power_{j,\gamma}(u; b) = \Phi\left(\frac{|a_j| w_{\gamma}(0) - u}{\sigma'_\gamma}\right)(1 + O(|a_j|^{-2})).
\end{equation}
\end{lemma}
\begin{proof}
By \eqref{eq:h-diff}, $h_{j,\gamma}'(v_j)=w_{\gamma}(0)$ is the maximum of $h_{j,\gamma}'(t)$ over $t\in \R$. The result then follows from Lemma 4 in \protect\citet{cheng2017multiple} with $z_\gamma(t)$ replaced by $z_\gamma'(t)$.
\end{proof}
By similar arguments in \protect\citet{cheng2017multiple} (see equation (20) therein), one can show that the random threshold $\tilde{u}_{\BH}$ converges asymptotically to the deterministic threshold
\begin{equation}
\label{eq:thresh-BH-fixed}
u^*_{\BH} = F_\gamma^{-1}\left(\frac{\alpha A_1}{A_1 + 2\E[\tilde{m}_{0,\gamma}(U(1))](1-2c\gamma A_1)(1-\alpha)}\right),
\end{equation}
where $\E[\tilde{m}_{0,\gamma}(U(1))]$ is given by \eqref{eq:Expect-LocalMax}. Since $\tilde{u}_{\BH}$ is random, similarly to the definition of $\FDR_{\BH, \gamma}(b)$, we define power in the BH procedure as
\begin{equation*}
\begin{split}
\Power_{\BH,\gamma}(b) &= \E \Bigg[ \frac{1}{J} \Bigg(\sum_{j\in I^+} \mathbbm{1}\left(
\tilde{T}^+_\gamma(\tilde{u}_{\BH}) \cap (v_j-b, v_j+b) \ne \emptyset \right) \\
&\qquad \qquad + \sum_{j\in I^-} \mathbbm{1}\left(
\tilde{T}^-_\gamma(\tilde{u}_{\BH}) \cap (v_j-b, v_j+b) \ne \emptyset \right) \Bigg)\Bigg].
\end{split}
\end{equation*}

\begin{theorem}
\label{thm:power}
Let conditions (C1) and (C2) hold.

(i) Suppose Algorithm \ref{alg:STEM} is applied with a fixed threshold $u>0$. Then
\[
\Power_\gamma(u; b) = 1 - O(a^{-2}).
\]

(ii) Suppose Algorithm \ref{alg:STEM} is applied with the random threshold $\tilde{u}_{\BH}$ \eqref{eq:thresh-BH-random}. Then
\[
\Power_{\BH,\gamma}(b) = 1 - O(a^{-2}+L^{-1/2}).
\]
\end{theorem}
\begin{proof}
The desired results follow from similar arguments for showing Theorem 5 in \protect\citet{cheng2017multiple}.
\end{proof}

\begin{example}[{\bf Gaussian autocorrelation model}]
\label{ex:Gaussian}
Let the noise $z(t)$ in \eqref{eq:signal+noise} be constructed as
\begin{equation}
\label{eq:noiseEx}
z(t) = \sigma \int_{\R} \frac{1}{\nu} \phi\left(\frac{t-s}{\nu}\right)\,dB(s), \qquad \sigma, \nu > 0,
\end{equation}
where $\phi$ is the standard Gaussian density, $dB(s)$ is Gaussian white noise and $\nu > 0$ ($z(t)$ is regarded by convention as Gaussian white noise when $\nu=0$). Convolving with a Gaussian kernel $w_\gamma(t) = (1/\gamma)\phi(t/\gamma)$ with $\gamma > 0$ as in \eqref{eq:mu-gamma} produces a zero-mean infinitely differentiable stationary ergodic Gaussian field $z_\gamma(t)$ such that
$$
z'_\gamma(t) = w'_\gamma(t) * z(t) = \sigma \int_{\R^N} \frac{-(t-s)}{\xi^2} \phi\left(\frac{t-s}{\xi}\right)\,dB(s), \quad \xi = \sqrt{\gamma^2 + \nu^2},
$$
with $\sigma'^2_\gamma = \sigma^2/(4\sqrt{\pi} \xi^3)$, $\lambda_{4,\gamma} = 3\sigma^2/(8\sqrt{\pi} \xi^5)$ and $\lambda_{6,\gamma} = 33\sigma^2/(16\sqrt{\pi} \xi^7)$. We have
\begin{equation}\label{eq:SNR}
\SNR_{j,\gamma} = \frac{a_j w_\gamma(0)}{\sigma'_\gamma} = \frac{\sqrt{2}|a_j|}{\sigma\pi^{1/4}}\left[\frac{(\gamma^2 + \nu^2)^{3/4}}{\gamma}\right].
\end{equation}
As a function of $\gamma$, the SNR has a local minimum at $\gamma^* = \sqrt{2}\nu$ and is strictly increasing for large $\gamma$. In particular, when $\nu=0$, it is strictly increasing in $\gamma$. Thus we generally expect the detection power to increase with $\gamma$ for $\gamma > \sqrt{2} \nu$. This will be confirmed in the simulations below. Note that for $\nu > 0$, the SNR is unbounded as $\gamma \to 0$, however in practice $\gamma$ cannot be too small: if the support of $w_\gamma$ becomes smaller than the sampling interval, then the derivative $\mu'_\gamma$ cannot be estimated.
\end{example}

\section{Simulation studies}
\label{sec:simulations}

\subsection{Performance of the dSTEM algorithm}
\label{simu:dSTEM}

Simulations were used to evaluate the performance and limitations of the dSTEM algorithm for signals $\mu(t)=a\lfloor t/d \rfloor$, where $t = 1, \ldots, L$, $L = 12000$, and signal strength $a\in \{1, 1.5 ,2\}$. Under this setting, the true change points are $v_j = jd$ for $j = 1, \ldots, L/d-1$, and the distance between neighboring change points is $d=100$. The noise is generated as the Gaussian process constructed in \eqref{eq:noiseEx} with $\sigma=1$ and varying $\nu$. Notice that the random error is white noise when $\nu=0$, and is autocorelated when $\nu > 0$. The smoothing kernels are $w_\gamma(t) = (1/\gamma)\phi(t/\gamma)\mathbbm{1}(t\in[-4\gamma, 4\gamma])$ for varying $\gamma$. The BH procedure was applied at FDR level $\alpha=0.1$ and the tolerance $b=5$. Results were averaged over 1,000 replications to simulate the expectations.

\begin{figure}[!h]
\centering
\includegraphics[width=\textwidth]{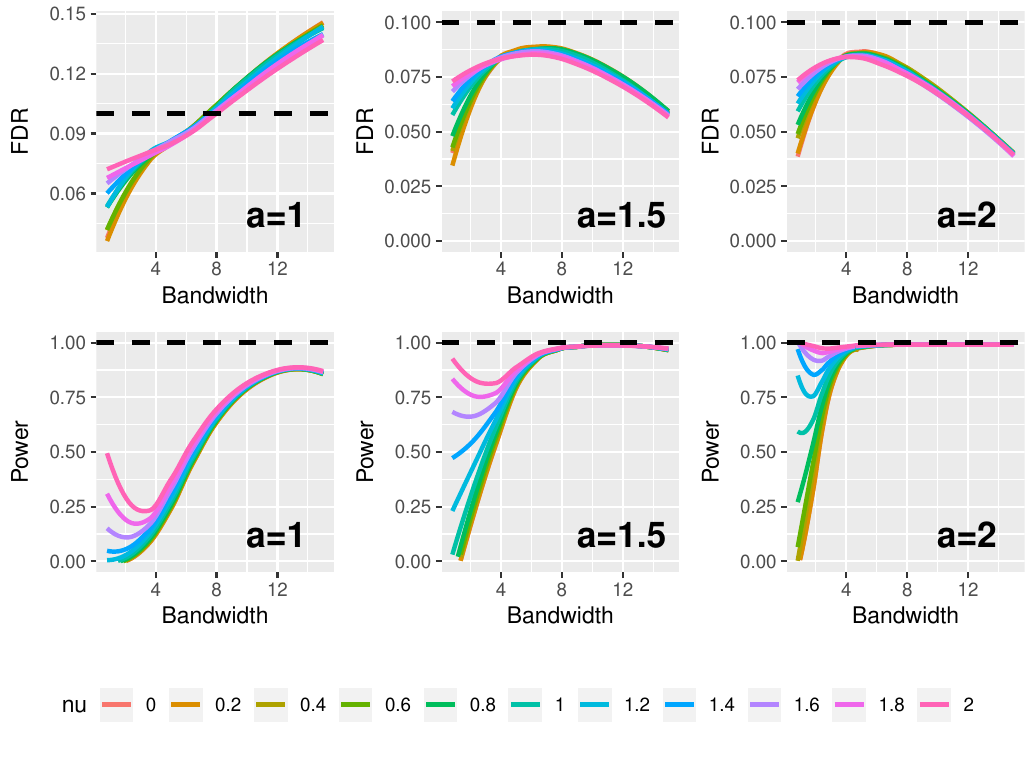}
\caption{The FDR (top) and power (bottom) vs. different combinations of the smoothness parameter $\nu$ (ranging from 0 to 2), the signal strength $a$ (taking values 1, 1.5 and 2) and the bandwidth $\gamma$ (ranging from 0.2 to 15). Here, the significance level $\alpha=0.1$, tolerance $b=5$ and $d=100$.}
\label{fig:fdrpow}
\end{figure}

The results of FDR and power are shown in Figure \ref{fig:fdrpow}. We see that for fixed $\gamma$ and $\nu$, as the strength of the signal $a$ increases, FDR will decrease while the power will increase; moreover, FDR is eventually controlled below the nominal level and the power tends to 1, which is consistent with Theorems \ref{thm:FDR} and \ref{thm:power}. For each fixed $a$, the power is seen to first decrease quickly and then increase again as $\gamma$ increases. This phenomenon coincides with the behavior of the SNR \eqref{eq:SNR} derived in Example \ref{ex:Gaussian}, predicting the power to decrease for $\gamma\le \sqrt{2}\nu$ and increase for $\gamma>\sqrt{2}\nu$. Meanwhile, if $a$ is moderate or large, the FDR is seen to first increase and then decrease as $\gamma$ increases, with the maximum of FDR still controlled below the nominal level. 

\begin{figure}[!h]
	\centering
	\includegraphics[width=\textwidth]{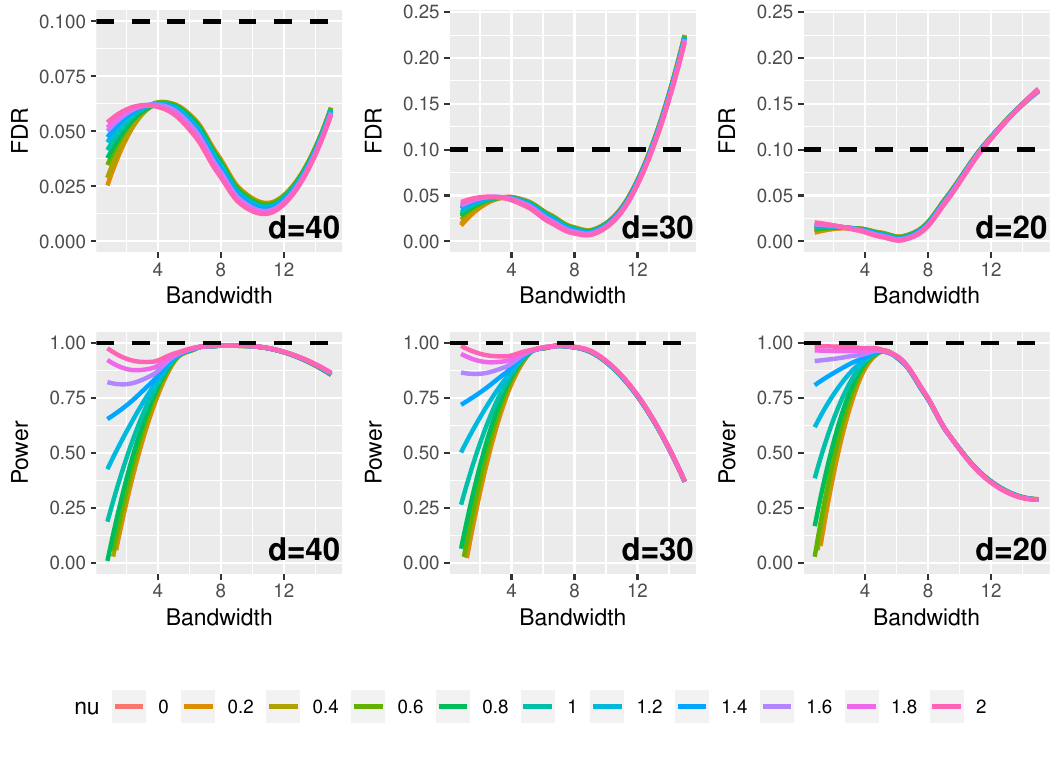}
	\caption{The FDR (top) and power (bottom) vs. different combinations of the smoothness parameter $\nu$ (ranging from 0 to 2), the width between neighboring change points $d$ (taking values 40, 30 and 20) and the bandwidth $\gamma$ (ranging from 0.2 to 15). Here, the significance level $\alpha=0.1$, tolerance $b=5$ and signal strength $a=1.5$.}
	\label{fig:fdrpowd}
\end{figure}

In the simulations in Figure \ref{fig:fdrpow} above, the distance of neighboring change points $d=100$ is large enough so that the kernel smoothing effectively affects only one change point at a time. However, if $d$ is small, then the kernel smoothing with large $\gamma$ may produce interference between neighboring change points, causing the power decrease. To illustrate this, we take the case where the signal strength is $a=1.5$ and perform simulations for FDR and power with $d=40$, $30$ and $20$. As shown in Figure \ref{fig:fdrpowd}, too large a $\gamma$ makes FDR increase and power decrease, due to the overlap between the kernel smoothing at neighboring change points. This phenomenon becomes more evident when $d$ is small ($d=20$). Theoretically, the neighboring interference would happen when $d$ is less than the support of the smoothing kernel, which is about $8\gamma$ in our Gaussian case here. In particular, we see that the turning point of the power, attaining almost its maximum, appears at around $\gamma=8$ for cases $d=40$ and $d=30$, while at around $\gamma=5$ for the case $d=20$. This suggests that $\gamma=d/4$ is a good choice of bandwidth when $d$ is not large.

\subsection{Choice of the bandwidth $\gamma$}

Figure \ref{fig:fdrpow} shows that the bandwidth $\gamma$ will greatly affect the performance of dSTEM.
We see that larger $\gamma$ tends to attain a smaller FDR and larger power. 
However, as shown in Figure \ref{fig:fdrpowd}, if $\gamma$ is too large, it will produce interference between neighboring change points and contamination error, thereby decreasing the power. Thus, the choice of $\gamma$ is critical for the performance of dSTEM. 

The optimal bandwidth $\gamma$ is the value that maximizes the power while controlling the FDR under the significance level. It is difficult to obtain the optimal $\gamma$ theoretically. However, in our model, the signal is a piecewise constant function, and it is possible to obtain the optimal $\gamma$ in practice by making more assumptions on the noise, such as in Example \ref{ex:Gaussian} using the Gaussian autocorrelation model. Figure \ref{fig:fdrpow} suggests that $\gamma$ should be chosen to be about 8 for weak signals, while it can be as small as 4 for strong signals, almost regardless of the noise autocorrelation. To avoid producing interference between neighboring change points, the minimal distance between change points $d=\inf_{j}|v_j-v_{j-1}|$ defined in \eqref{eq:a-v} should be large. On the other hand, if $d$ is not large, then we can choose the bandwidth $\gamma$ to be about $d/c$ where the effective support of the smoothing kernel is $\pm c\gamma$. This has been shown in simulations in Figure \ref{fig:fdrpowd} where $c=4$ and $\gamma=d/4$.

\subsection{Comparison with algorithm FDRSeg}

As mentioned in Section \ref{sec:intro}, FDRSeg is the newest method which can control FDR for change point detection. 
In this subsection, we compare the performance of our method dSTEM with the algorithm FDRSeg. First, it is worth mentioning that our method is mainly designed for autocorrelated random noise, while FDRSeg requires independent and identically distributed random error, which is just a special case of our method ($\nu=0$). Note that FDRSeg contains only one parameter $\alpha_{\text{F}}$, which controls the theoretical upper bound of FDR at $2\alpha_{\text{F}}/(1-\alpha_{\text{F}})$. However, they suggest that in practice their method should give $\text{FDR} \leq \alpha_{\text{F}}$. Thus, we let $\alpha_{\text{F}}$
be 0.05 and 0.1.

Table \ref{tab:white} shows the realized FDR and detection power under independent noise situation.
We see that for small signal $a=1$, dSTEM can almost control FDR and its power is $84.8\%$ when $\gamma=12$; while FDRSeg can attain a little larger power, but it is hard to control FDR when $\alpha=0.1$.
For larger signal $a$, both two methods can control FDR and attain a similar large power.
Table \ref{tab:nowhite} shows the results under the situation of autocorrelated noise ($\nu=1$).
In this case, the performance of dSTEM is nearly the same as that in independent scenario, while FDRSeg tends to estimate a large number of change points, leading to a large FDR, which means it can hardly control FDR.

\begin{table}[!h]
  \centering
  \small
  \caption{Performance comparison of dSTEM and FDRSeg under independent noise.}
  \label{tab:white}
  \begin{threeparttable}      
    \begin{tabular}{c|ccc|ccc} 
      \toprule      
     \multicolumn{1}{c}{} & \multicolumn{3}{c}{\textbf{dSTEM}} & \multicolumn{3}{c}{\textbf{FDRSeg}}\\
      \midrule
      \multirow{5}*{$a=1$}
      &$\gamma$ & FDR & Power & $\alpha_{\text{F}}$ \tnote{a} & FDR & Power\\
      &9 &0.113 & 0.723 & \multirow{2}*{0.05} & \multirow{2}*{0.119} & \multirow{2}*{0.840} \\
      &10 &0.117 & 0.781 & & & \\
      &11 &0.124 & 0.820 & \multirow{2}*{0.1} & \multirow{2}*{0.155} & \multirow{2}*{0.865} \\
      &12 &0.131 & 0.848 & & &\\
      \hline
      \multirow{5}*{$a=1.5$}
      &$\gamma$ & FDR & Power & $\alpha_{\text{F}}$ & FDR & Power\\
      &6 &0.091 &0.896 & \multirow{2}*{0.05} & \multirow{2}*{0.033} & \multirow{2}*{0.943} \\
      &7 &0.089 &0.943 &  & & \\
      &8 &0.088 &0.965 & \multirow{2}*{0.1} & \multirow{2}*{0.090} & \multirow{2}*{0.957} \\
      &9 &0.083 &0.974 &  & & \\
      \hline
      \multirow{5}*{$a=2$}
      &$\gamma$ & FDR & Power & $\alpha_{\text{F}}$ & FDR & Power\\
      &4 &0.085 &0.932 &\multirow{2}*{0.05} & \multirow{2}*{0.008} & \multirow{2}*{0.971}  \\
      &5 &0.088 &0.978  & & &  \\
      &6 &0.082 &0.987  &\multirow{2}*{0.1} & \multirow{2}*{0.049} & \multirow{2}*{0.983}  \\
      &7 &0.085 &0.989 & & & \\
      \bottomrule   
    \end{tabular} 
    \begin{tablenotes}\footnotesize
     \item [a] the only parameter in FDRSeg contralling the theoretical upper bound of FDR to $2\alpha_{\text{F}}/(1-\alpha_{\text{F}})$
    \end{tablenotes}
  \end{threeparttable}
\end{table}  

\begin{table}[!h]
  \centering
  \small
  \caption{Performance comparison of dSTEM and FDRSeg under autocorrelated noise ($\nu=1$).}
    \label{tab:nowhite}
    \begin{tabular}{c|ccc|ccc} 
      \toprule
      \multicolumn{1}{c}{} & \multicolumn{3}{c}{\textbf{dSTEM}} & \multicolumn{3}{c}{\textbf{FDRseg}}\\
      \midrule
      \multirow{5}*{$a=1$}
      &$\gamma$ & FDR & Power & $\alpha_{\text{F}}$ & FDR & Power\\
      &9 &0.112 &0.733 & \multirow{2}*{0.05} & \multirow{2}*{0.808} & \multirow{2}*{1.000}\\
      &10 &0.118 &0.792  & & & \\
      &11 &0.127 &0.827 &\multirow{2}*{0.1} & \multirow{2}*{0.815} & \multirow{2}*{1.000} \\
      &12 &0.134 &0.851  & & & \\
      \hline
      \multirow{5}*{$a=1.5$}
      &$\gamma$ & FDR & Power & $\alpha_{\text{F}}$ & FDR & Power\\
      &6 &0.088 &0.908 & \multirow{2}*{0.05} & \multirow{2}*{0.796} & \multirow{2}*{1.000} \\
      &7 &0.086 &0.949  & & & \\
      &8 &0.086 &0.968 & \multirow{2}*{0.1} & \multirow{2}*{0.802} & \multirow{2}*{1.000} \\
      &9 &0.084 &0.976  & & & \\
      \hline
      \multirow{5}*{$a=2$}
      &$\gamma$ & FDR & Power & $\alpha_{\text{F}}$ & FDR & Power\\
      &4 &0.084 &0.952 & \multirow{2}*{0.05} & \multirow{2}*{0.785} & \multirow{2}*{1.000} \\
      &5 &0.083 &0.980  & & & \\
      &6 &0.083 &0.988 & \multirow{2}*{0.1} & \multirow{2}*{0.795} & \multirow{2}*{1.000}\\
      &7 &0.081 &0.990  & & & \\ 
      \bottomrule   
    \end{tabular}
\end{table}

\section{Data example}
\label{sec:data}

\subsection{Magnetometer sensor readings}
In the field of mobile security, two-factor authentication/verification is of great importance,  which is an extra layer of security of your mobile device, such as smartphones, wearable, and smart home devices, designed to ensure that you are the only person who can unlock your device, even if someone knows your password.
In recent years, gesture based key establishment is a popular topic in communication security and computer science \protect\citep{tan2010dynamic,simao2016natural,kiliboz2015hand}. 

Modern mobile devices embedded with various motion sensors including accelerometer, gyroscope and magnetometer are used to measure and record the gesture performing process.
Obtaining accurate readings of magnetometer is the foundation of magnetometer baesd research. However, during data collection, there are always noises which might be caused by hardware imperfection, manipulation error or sensitivity of the sensor. Particularly, finding the start point and associated end point for each gesture is a big challenge. The goal of this analysis is to find such change points.

In this paper, the data was collected from an experiment where several simple gestures, for example, shaking the smartphone in different directions and at different speeds, were designed. In particular, the smartphone defines a coordinate system of the embedded magnetometer, which is shown in Figure \ref{fig:mobile}. The magnetometer can record the speed of smartphone movement as the readings along X, Y and Z axes. In our case, we will only show the results of the readings on X-axis, since readings along other two axes could be processed similarly. 

\begin{figure}[!h]
\begin{center}
\includegraphics[trim=0 0 0 0,clip,width=3.2in]{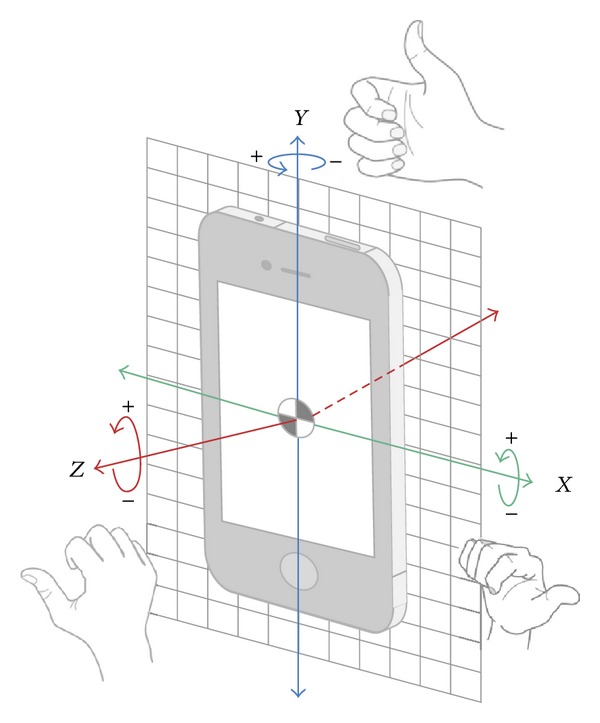}
\caption{The coordinate system of the magnetometer embedded in a smartphone.}
\label{fig:mobile}
 \end{center}
 \end{figure}

\begin{figure}[!h]
\begin{center}
\includegraphics[trim=0 0 0 0,clip,width=4.7in]{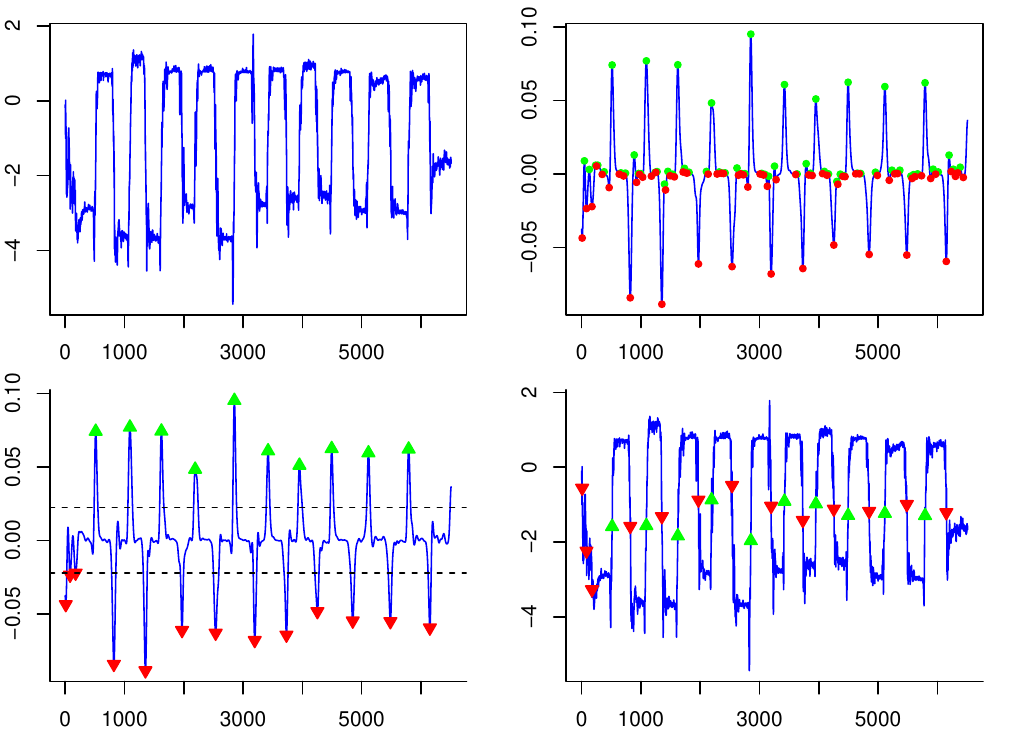}
\caption{Magnetometer example. Top left: Observed data. Top right: Estimated first derivative of the smoothed data, and local maxima (green), local minima (red), and significant height threshold (black dashed line). Bottom left: The first derivative , and the detected positive (green upward triangle) and negative (red downward triangle) change points. Bottom right: The observed data and its change points.}
\label{fig:sensor}
 \end{center}
 \end{figure}

In this data analysis, the sample size is $n = 6510$ and we use the same procedure as in example \ref{data:CGH} except $\gamma=18$,
because the interval between neighboring change points is narrow so that the bandwidth cannot be too large to avoid interference.
Figure \ref{fig:sensor} shows results of the detected change points. Due to the measurement error and magnetic-field interference, the real underlying data will be interfered by slight fluctuations, leading to lots of (349) local maxima and minima, as shown in Figure \ref{fig:sensor} (top right). However, despite that, our method can still find the true change points, whose number is actually not large, as shown in the bottom left panel. In the bottom right panel, it is obvious that the start points and associated end points are very well detected. 

Figure \ref{fig:data_seg} shows the results of FDRSeg, it is obvious that FDRSeg estimates too many (1609) change points, which is nonsense and consistent with the simulation results under autocorrelated noise.

\begin{figure}[!h]
\begin{center}
\includegraphics[trim=0 0 0 0,clip,width=4.7in]{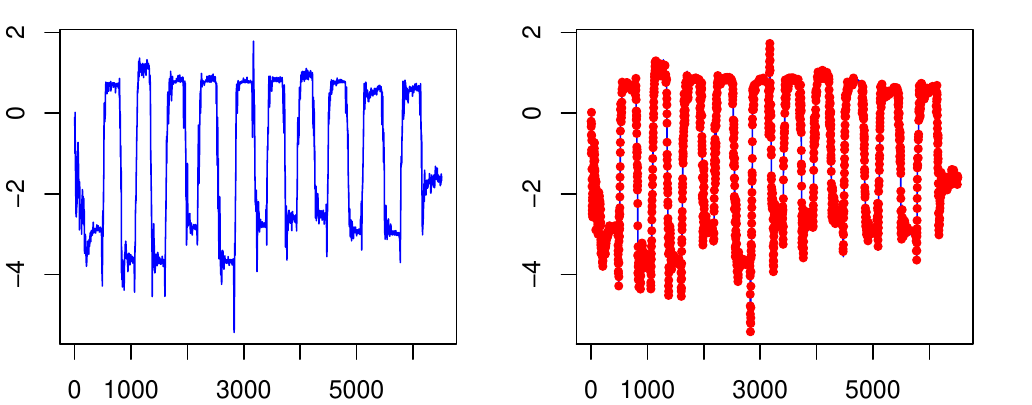}
\caption{Magnetometer example. Left : Observed data. Right: Detected change points (red) by FDRSeg.}
\label{fig:data_seg}
 \end{center}
 \end{figure}
\subsection{Array-CGH data}
\label{data:CGH}
Array-based comparative genomic hybridization (array-CGH) is a high-throughput high-resolution technique used to evaluate changes in the number of copies of alleles at thousands of genomic loci simultaneously. The output is often called Copy Number Variation (CNV) data. Changes in copy number are represented by segments whose mean is displaced with respect to the background. To detect these changes, it is costumary to search for change points along the genome.

In this paper, we apply our method to chromosome 1 of tumor sample \#18 from the dataset of \protect\citet{hsu2005denoising} and \protect\citet{loo2004array}. This sample is one of 37 formalin-fixed breast cancer tumors in that dataset and it was chosen for its visual appeal in the illustration of our method. The data in chromosome 1 of tumor sample \#18 consists of 968 average copy number reads mapped onto 968 unequally spaced locations along the chromosome. For simplicity, the data was analyzed ignoring the gaps in the genomic locations. Figure \ref{fig:data} (top left) shows the data with spacings between reads artificially set to 1. Note that ignoring the spacings does not affect the presence or absence of change points.

To analyze the data, the dSTEM algorithm was applied with a truncated Gaussian smoothing kernel $w_\gamma(t) = (1/\gamma)\phi(t/\gamma)\mathbbm{1}(t\in[-4\gamma, 4\gamma])$ with $\gamma=10$. The bandwidth was chosen not too large in order to avoid interference between neighboring change points. Figure \ref{fig:data} (top right) shows the estimated first derivative \eqref{eq:conv-diff}. Figure \ref{fig:data} (bottom left) marks 19 local maxima (green) and 19 local minima (red).

P-values corresponding to local maxima and minima were computed according to \eqref{eq:pval} using the distribution \eqref{eq:distr}. The required parameters ${\sigma'_\gamma}^2={\rm Var}(z_\gamma'(t))$, $\lambda_{4, \gamma}={\rm Var}(z_\gamma^{''}(t))$, $\lambda_{6, \gamma}={\rm Var}(z_\gamma^{(3)}(t))$
were estimated empirically from the estimated first, second and third derivatives over the observed data sequence. However, the empirical variances were computed using truncated averages instead of regular averages in order to avoid bias from the extreme derivatives at the change points without assuming their presence or location in advance. The BH algorithm was applied to the 38 p-values FDR level 0.2, yielding a p-value significance threshold of $4.42\times 10^{-4}$. The corresponding absolute height threshold of 0.089 is marked as dashed lines in Figure \ref{fig:data} (bottom left). The significant peaks are plotted on the original data in Figure \ref{fig:data} (bottom right) with a location tolerance of $b=2$ for visual reference.

\begin{figure}[t]
\begin{center}
\begin{tabular}{cc}
\includegraphics[trim=0 0 0 0,clip,width=2.2in]{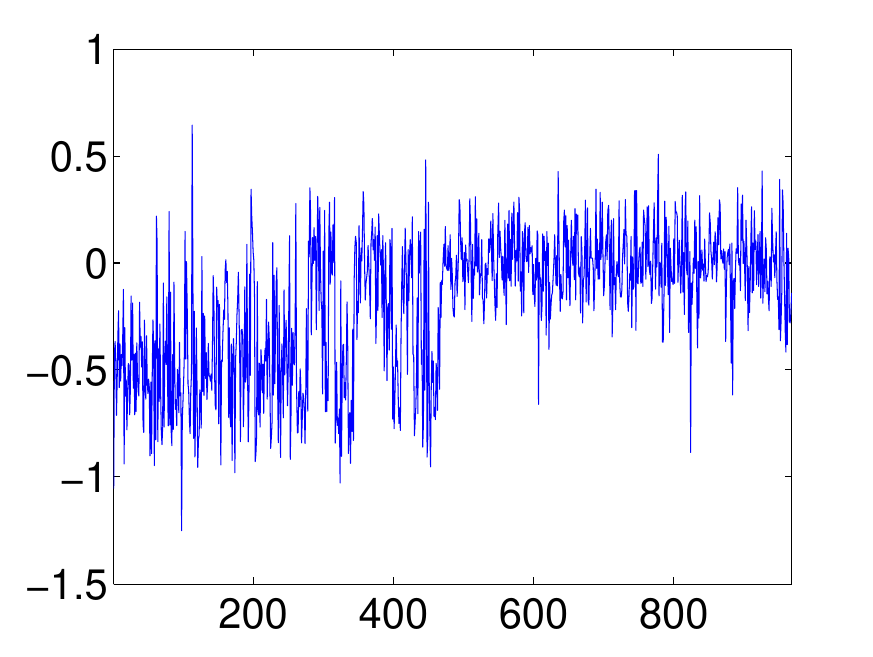} &
\includegraphics[trim=0 0 0 0,clip,width=2.2in]{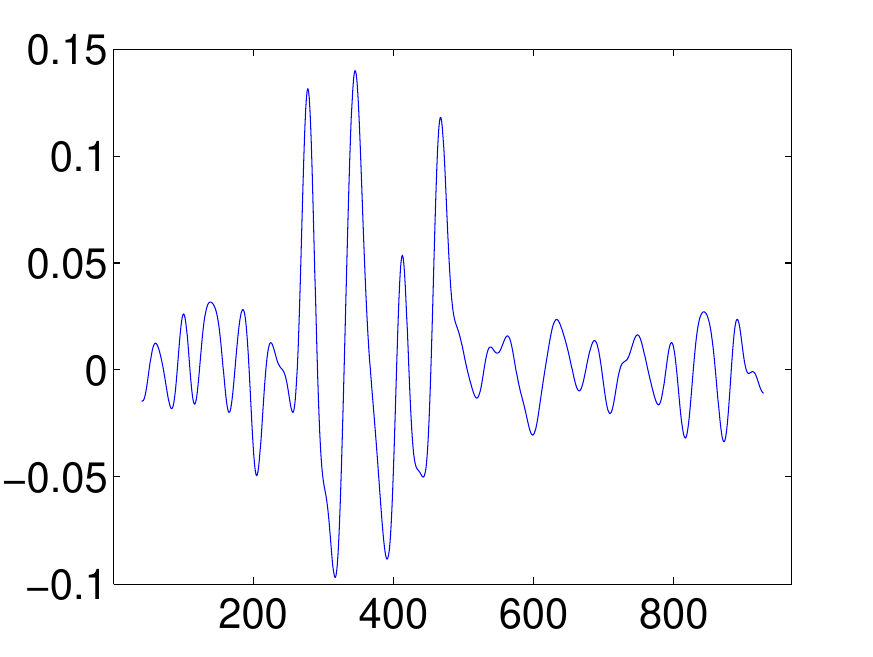}\\
\includegraphics[trim=0 0 0 0,clip,width=2.2in]{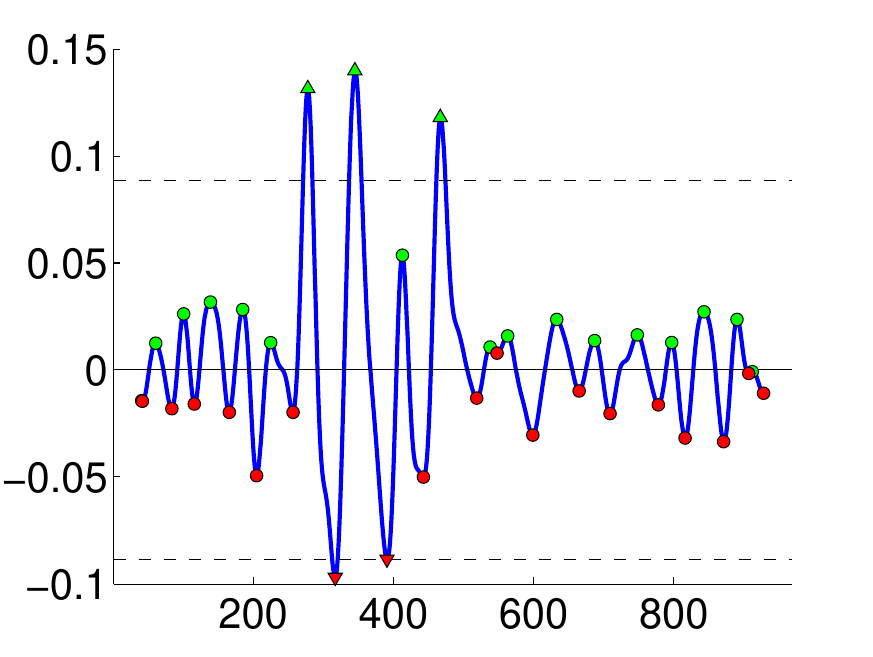} &
\includegraphics[trim=0 0 0 0,clip,width=2.2in]{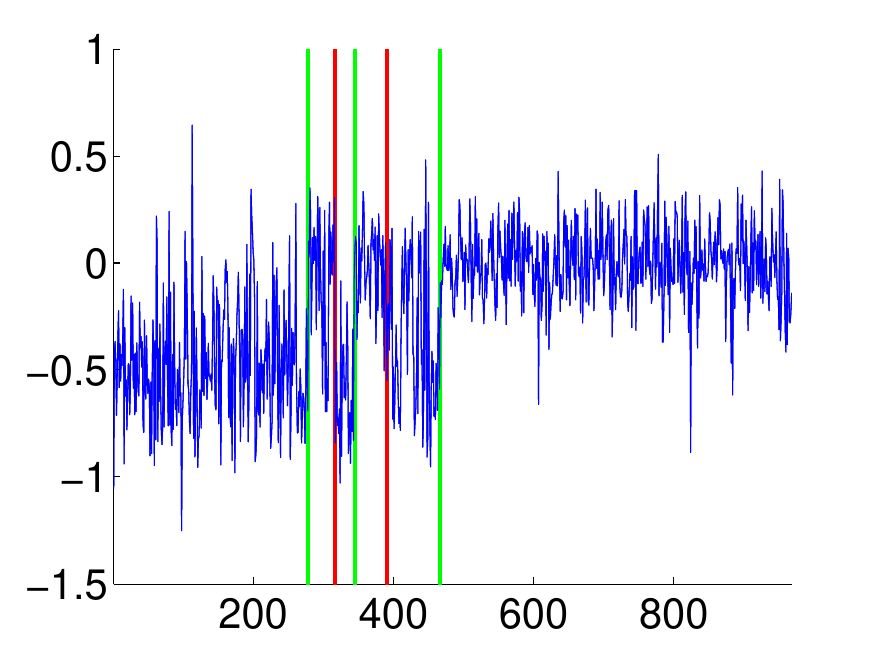}\\
\end{tabular}
\caption{\label{fig:data} Data example. Top left: Observed data. Top right: Estimated first derivative. Bottom left: Local maxima (upward triangles), local minima (downward triangles) of the estimated first derivative, and significance height threshold (black dashed line). Bottom right: The detected positive (green) and negative (red) change points.}
 \end{center}
 \end{figure}

\section{Discussion}
\label{sec:discussion}


\subsection{Increasing and decreasing change points} In this paper, we combined both local maxima and minima of the derivative as candidate peaks, and then applied a multiple testing procedure to find a uniform threshold (in absolute value) for detecting all change points. This approach is sensible when the distributions (number and height) of true increasing and decreasing change points are about the same. Alternatively, different thresholds for detecting increasing and decreasing change points could be found by applying separate multiple testing procedures to the sets of candidate local maxima and local minima. While we applied the BH algorithm to control FDR, in principle other multiple testing procedures may be used to control other error rates.

\subsection{The smoothing bandwidth} A natural and important question is how to choose the smoothing bandwidth $\gamma$. We can see that either a small $\gamma$ (if the noise is highly autocorrelated) or a relatively large $\gamma$ (if the noise is less autocorrelated) is preferred in order to increase power, but only to the extent that the smoothed signal supports $h'_{j,\gamma}(t)$ have little overlap and that detected change points are not displaced by more than the desired tolerance $b$ (recall that the value of $b$ is not used in the dSTEM algorithm itself, but it may be determined by the needs of the specific scientific application). Considering the Gaussian kernel to have an effective support of $\pm c\gamma$, a good value of $\gamma$ may be about $\min(b,d/c)$, where $d$ is the separation between change points. Since the location of the change points is unknown, a more precise optimization of $\gamma$ may require an iterative procedure. Moreover, if some change points are close together and others are far apart, an adaptive bandwidth may be preferable. We leave these as problems for future research.

\section*{Supplementary Material}

An R package named ``dSTEM", for performing the dSTEM algorithm \ref{alg:STEM} for change point detection, is available in R cran.


\printbibliography

\begin{quote}
	\begin{small}
		
		\textsc{Dan Cheng} and \textsc{Zhibing He}\\
		School of Mathematical and Statistical Sciences \\
		Arizona State University\\
		900 S. Palm Walk\\
		Tempe, AZ 85281, U.S.A.\\
		E-mail: \texttt{cheng.stats@gmail.com; zhibingh@asu.edu}

		\vspace{.1in}
		
		\textsc{Armin Schwartzman}\\
		Division of Biostatistics and Bioinformatics and \\
		Halicio\v{g}lu Data Science Institute \\
		University of California San Diego\\
		9500 Gilman Dr.\\
		La Jolla, CA 92093, U.S.A.\\
		E-mail: \texttt{armins@ucsd.edu}				
	\end{small}
\end{quote}

\end{document}